\documentclass[11pt,leqno]{article}

\usepackage{amssymb,amsmath,amsthm,mysects,url,rotating}
 \usepackage{graphicx,epsfig}
 \usepackage{xcolor,graphicx}
  \usepackage{hyperref}
\newcommand{\n}{\noindent}

\newcommand{\vp}{\varepsilon}
\newcommand{\bb}[1]{\mathbb{#1}}
\newcommand{\cl}[1]{\mathcal{#1}}

\newcommand{\ovl}{\overline}

\theoremstyle{plain}
\newtheorem{thm}{Theorem}[section]
\newtheorem{lem}[thm]{Lemma}

\newtheorem{pro}[thm]{Proposition}

\newtheorem{cor}[thm]{Corollary}

\theoremstyle{definition}

\newtheorem{dfn}[thm]{Definition}

\theoremstyle{remark}
\newtheorem{rem}[thm]{Remark}

\numberwithin{equation}{section}

\def\tilde{\widetilde}

\renewcommand{\tilde}{\widetilde}

\def\R{\bb R}

\def\C{\bb C}

\def\F{\bb F}
\def\P{\bb P}
\def\T{\bb T}

\def\d{\delta}

\def\N{\bb N}

\def\F{\bb F}
\def\P{\bb P}
\def\T{\bb T}
\def\nl{\nolimits}
\def\d{\delta}

\def\tilde{\widetilde}

\renewcommand{\tilde}{\widetilde}

\def\R{\bb R}
\def\Z{\bb Z}

\def\C{\bb C}

\def\N{\bb N}
\def\P{\bb P}

\def\T{\bb T}

\def\hat{\widehat}
\def\nl{\nolimits}
\begin{document}
\def\d{\delta}

  \def\tr{{\rm tr}}
 \def\y{\varphi}
\title{Completely Sidon sets in   $C^*$-algebras}

\author{by\\
 Gilles  Pisier\\
Texas A\&M University and UPMC-Paris VI}

 \maketitle
 \begin{abstract}  
 A sequence in a $C^*$-algebra $A$
  is called
 completely Sidon if its span in $A$ is completely isomorphic
 to the operator space version of the space $\ell_1$
(i.e. $\ell_1$ equipped with its maximal operator space structure).
The latter can also be described  as the span of the
free unitary generators in the (full) $C^*$-algebra of the free group
$\F_\infty$ with countably infinitely many generators. 
Our main result is a generalization
to this context of Drury's classical  theorem
stating that  Sidon sets are stable under finite unions.
In the particular case when $A=C^*(G)$ the (maximal) $C^*$-algebra of a discrete group $G$, we recover the non-commutative (operator space) version of Drury's
theorem that we recently proved.
We also give several non-commutative generalizations
of   our recent work on uniformly bounded orthonormal systems
to the case of von Neumann algebras equipped with normal
faithful tracial states.
 \end{abstract}  
 
 MSC Classif. 43A46, 46L06
 
 Recently,
 following the impulse of Bourgain and Lewko \cite{BoLe},
 we studied in \cite{Pi3}  the uniformly bounded orthonormal systems
 that span in $L_\infty$ a subspace isomorphic to $\ell_1$  
 by the basis to basis equivalence, and we called them Sidon sequences
 in analogy with the case of  characters on compact abelian groups.
One of the main results in  \cite{Pi3}  says that
if a uniformly bounded orthonormal system  $\{\psi_n\}$ in $L_2$ of a probability space
$(\Omega,\P)$ is the union $\{\psi^1_n\} \cup \{\psi^2_n\}$
of two Sidon sequences,
then the sequence $\{\psi_n\otimes \psi_n\otimes \psi_n\otimes \psi_n\}$
or simply $\{ \psi_n^{\otimes^4} \}$ is Sidon in $L_\infty(\Omega^4,\P^4)$.
Our goal in this paper is to generalize this result
to sequences in a non-commutative $C^*$-algebra.
The central ingredient of
our method in \cite{Pi3} is the spectral decomposition of the Ornstein-Uhlenbeck semigroup  for a Gaussian measure on $\R^n$.
Since this has all sorts of non-commutative analogues,
it is natural to try to extend the  results  of \cite{Pi3}
to non-commutative von Neumann or  $C^*$-algebras 
in place of $L_\infty$.
In \cite{Pi3} ``subgaussian" and ``randomly Sidon" sequences play an important role.
Although the non-commutative analogue
of a subgaussian system is not  clear  
(see however Remark \ref{Rf}),
and  that of ``randomly Sidon set" eludes us for the moment,
we are able 
in the present paper to extend several of the main results
of \cite{Pi3}, in particular
we recover an analogue of Drury's famous union theorem
for Sidon sets in groups. 
In the commutative case the fundamental example of
Sidon set is the set formed of the canonical generators
in the group $\Z_\infty$ formed of all the
finitely supported functions $f: \N \to \Z$. 
This is sometimes referred to as the free Abelian
group with countably  infinitely many generators.
The dual of the discrete group $\Z_\infty$
is the compact group $\T^\N$, and the von Neumann algebra
of $\Z_\infty$ can be identified with $L_\infty(\T^\N)$.
The analogue of this for our work
is the free group 
$\F_\infty$ with countably  infinitely many generators,
and its von Neumann algebra $M_{\F_\infty}$.
In the commutative case the generators of 
$\Z_\infty$ correspond in $L_\infty(\T^\N)$
to independent random variables uniformly distributed over $\T$.
In classical Sidon set theory, the associated Riesz product
plays a crucial role, because of its special interpolation property
derived from its spectral decomposition.
In more modern approaches,  these variables are replaced by
standard i.i.d.
gaussian random variables, and in \cite{Pi3}
the  ``spectral/interpolation property" of Riesz products
is replaced by the spectral expansion of the Ornstein-Uhlenbeck
semigroup $T_\vp$, i.e. the one that multiplies a multivariate Hermite function
of degree $d$ by $\vp^d$ (here $0\le \vp\le 1$). 
Equivalently this is obtained by second quantization
applied to $\vp \times \text{Identity}$ on the \emph{symmetric} Fock space.
In our new setting, the proper analogue
comes from Voiculescu's free probability theory (cf. \cite{VDN}),
where the analogues of gaussian variables are
operators on the \emph{full} Fock space.
 Not surprisingly, the techniques we use
 come from non-commutative probability, in connection with operator space theory
 for which we refer to \cite{P4}.
 
 By definition an operator space is a subspace $E\subset A$ of a $C^*$-algebra,
 and the operator space structure (o.s.s. in short) on $E$ consists of the sequence of norms
 induced  on $M_n(E)$ by $M_n(A)$, where $M_n(E)$ denotes the space
 of $n\times n$-matrices with entries in $E$, and  $M_n(A)$ is equipped
 with its natural norm as a $C^*$-algebra.

Let $(e_n)$ denote the canonical basis of the Banach space $\ell_1$
of absolutely summable complex sequences. 
The space $\ell_1$ is equipped with a special o.s.s.
called the maximal one. The latter operator space structure  is   induced 
   by the isometric embedding
    $\ell_1\subset C^*(\F_\infty)$ taking $e_n$ to the
   $n$-th
free unitary generator in the maximal $C^*$-algebra $C^*(\F_\infty)$ (see \cite[p. 183]{P4}).
See   \cite[\S 3]{P4} for more information and references.

We recall that the algebraic tensor product $A\otimes B$ of two $C^*$-algebras
can be equipped with a minimal and a maximal $C^*$-norm, which after completion
produce the $C^*$-algebras $A\otimes_{\min} B$ and $A\otimes_{\max} B$.
If either $A$ or $B$ is commutative (or nuclear)
then $A\otimes_{\min} B=A\otimes_{\max} B$ isometrically.

It is well known that the linear maps between  $C^*$-algebras that are compatible with
the minimal tensor products  are the completely bounded (c.b. in short) ones, see \cite{ER,Pa2,P4}.
We should emphasize
that  the analogous maps for the maximal tensor products are the decomposable ones
(see \eqref{e8} and \eqref{e8'} below) for which we extensively use Haagerup's results in \cite{Haa}. 
See   \cite[\S 11]{P4} or \cite{BO} for more background. 

The natural
non-commutative generalizations of the notions
in \cite{Pifz} are like this:
\begin{dfn}
A bounded sequence $(\psi_n)$ 
  in a $C^*$-algebra $A$
 is called completely Sidon if the
 mapping taking $e_n$ to $\psi_n$
 is a complete  isomorphism
   when  $\ell_1$ is equipped with its maximal o.s.s..
    \\
The sequence $(\psi_n)$ is called completely $\otimes_{\max}^k$-Sidon
 in $A$
 if the sequence $(\psi_n\otimes \cdots\otimes \psi_n)$  ($k$-times)
 is completely Sidon when viewed as sitting in $A \otimes_{\max} \cdots\otimes_{\max} A$ ($k$-times).\end{dfn}
 When $A$ is commutative, we simply say that
 $(\psi_n)$ is     $\otimes^k$-Sidon (and when $k=1$
 we just call it  Sidon), thus recovering the terminology in \cite{Pi3}.

 Let $({\cl M},\tau)$ be a non-commutative tracial probability space, i.e.
 a von Neumann algebra equipped with a faithful normal tracial state.
Toward the end of this paper we   reach our (already announced) goal: we show
 that
if 
$\{\psi_n\}\subset L_2({\cl M},\tau)$ 
is an orthonormal system that is uniformly bounded in ${\cl M}$
and is the union 
$\{\psi^1_n\} \cup \{\psi^2_n\}$
of two completely Sidon sequences
 then $\{\psi_n\}$ is completely $\otimes_{\max}^4$-Sidon.
 
 One difficulty is the apparent lack of
a suitable non-commutative analogue
of the notion of subgaussian sequence, that is crucially used in
  \cite{Pi3},
  as well as that of a sequence ``dominated" by   gaussians.
  We say that sequence $(x^1_n)$ in an $L_1$-space is dominated by another one
  $(x^2_n)$ if there is a bounded linear map $v: L_1 \to L_1$ taking $x^2_n$ to $x^1_n$ for all $n$.
  In the non-commutative case, we use the same definition but c.b. maps are not enough, we must consider decomposable maps
   $v: {\cl M}_* \to {\cl N}_*$ between non-commutative $L_1$-space,
   i.e. preduals of von Neumann algebras ${\cl M,\cl N}$. By $v$ decomposable we just mean
   that the adjoint $v^*:\cl N \to \cl M$ is decomposable as a linear combination of c.p. maps. 
   When ${\cl M,\cl  N}$ are commutative any bounded linear map between them is decomposable,
   but in general it is not so.
   
  A key point in the commutative setting of \cite{Pi3} is   that if 
  a uniformly bounded orthonormal sequence    $\{\psi_n\}$ in $A=L_\infty(\P)$
 is  Sidon with constant $C$,
  then  there is a 
   biorthogonal   sequence $(y_n)$ in $A^*$ that is dominated by an i.i.d. bounded sequence $(x_n)$ (and a fortiori dominated by i.i.d. gaussians).
   The proof is very simple:  the mapping $u: \text{span}\{\psi_n\} \to L_\infty(\T^\N)$
   taking $\psi_n$ to the $n$-coordinate $z_n$ on $\T^\N$
   has norm $\le C$, it extends to a mapping $\tilde u : A \to  L_\infty(\T^\N)$
   with the same norm, then if we set  $x_n=\bar z_n$ viewed as an element of 
   $L_1(\T^\N)$, $(y_n)$ defined by $y_n= (\tilde u)^* (x_n)$ does the job. 
   
A second point  in \cite{Pi3} is that if a  sequence in $L_1(\P)$  is dominated by  a sequence 
  such as $(x_n)$
then any bounded sequence in $L_\infty(\P)$ that is biorthogonal to it
must be $\otimes^2$-Sidon.

In the non-commutative setting, $L_\infty(\P)$ is replaced by $({\cl M},\tau)$  and
$L_\infty(\T^\N)$ by
$M=M_{\F_\infty}$
equipped with its usual trace $   \tau_{\F_\infty}$.  
The non-commutative version of 
 $(x_n)$ is  the sequence $(y_n)$
 formed of variables each one having the same distribution as $\bar z_n$ or $z_n$
(i.e. normalized Haar measure on $\T$)
but instead of stochastic independence
we require freeness.
More formally we  
take for $y_n$ the element of $ M _*$
associated to the $n$-th free generator of ${\F_\infty}$.
Equivalently   $(y_n)$ can be any  free Haar unitary sequence 
in the sense of \cite{VDN}. We could use just as well any free semicircular (also called ``free gaussian") sequence
in Voiculescu's sense. More generally
we call ``pseudo-free" (see Remark \ref{ps}) any sequence
that is equivalent in a suitable sense (see Definition \ref{dom}) to such a sequence $(y_n)$.  Surprisingly, in this setting there is no need to distinguish
between the gaussian and i.i.d. unimodular case, because
free gaussian variables, unlike the gaussian ones, are uniformy bounded.

 The  non-commutative analogue 
 of the preceding  two points can be described schematically like this:
 \begin{thm}\label{t0}
 Assume $(\psi_n)$   completely  Sidon
 in $A$. Then $(\psi_n)$ admits
  a biorthogonal sequence $(\varphi_n)$ in $A^*$
 that is dominated by $(y_n)$.
 Moreover,    any bounded sequence $(\psi'_n)$
 in $A$ that is biorthogonal  to
 a sequence  dominated by $(y_n)$ is
 completely $\otimes_{\max}^2$-Sidon
 in $A$.
 \end{thm}
 See \S \ref{sf} for the proof.
This is particularly useful in the case
$A=C^*(G)$ ($G$  a discrete group)
when $\psi_n=U_G(t_n)$,  the $t_n$'s
being distinct elements of $G$, and $U_G$ being  the universal unitary representation of $G$.
We   say that the set $\Lambda=\{t_n\}$ is   completely Sidon  
in $G$ when $(U_G(t_n))$ is so in $A=C^*(G)$.
In this case 
completely $\otimes_{\max}^k$-Sidon
 in $A$ automatically implies 
completely  Sidon,
therefore ``completely  Sidon" is equivalent
to  having a biothogonal sequence dominated by $(x_n)$.
 
We deduce from this in Corollary \ref{coru} that the union of two 
completely  Sidon subsets of $G$ is completely  Sidon.
This reduces to show that the 
union $\Lambda=\{t_n\}$ of the two completely  Sidon sets
is such that,  in the group $G^2$, $\{(t_n,t_n)\}$ has a biothogonal sequence dominated by $(y_n)$. Indeed, the preceding Theorem
then tells us that $\{(t_n,t_n,t_n,t_n)\}$
is completely  Sidon in $G^4$ and this clearly is the same
as saying $\Lambda=\{t_n\}$ is completely  Sidon.
Note that while    the property ``dominated by $(x_n)$"
is preserved by the union of two sequences with it (see Remark \ref{R9}),
when dealing with   a disjoint union
$\Lambda_1\cup \Lambda_2\subset A$
we have to find a way around the following
 difficulty :    
the union
of a system biorthogonal to  $\Lambda_1$ with one
biorthogonal to  $\Lambda_2$ is not necessarily
biorthogonal to $\Lambda_1\cup \Lambda_2$. 
This explains why we pass to $G^2$. A similar difficulty arises
for a general $A$.
This point leads us to conclude that the union is  completely
$\otimes_{\max}^k$-Sidon only for $k=4$ when we would hope to find
$k=2$
(see
 the proof of Theorem
\ref{t2}).

The proof of Theorem \ref{t0} reduces to a special case
that we prove in Corollary \ref{key2}, namely the case 
when
$A=M$  
and $(\psi_n)=(y_n)$. 
This is analogous to the commutative result proved in \cite{Pi3}:
any  subgaussian sequence in $L_1(\P)$ (in particular  the above sequence $(\bar z_n)$ in $L_1(\T^\N)$) 
is such that
any bounded biorthogonal sequence in $L_\infty(\P)$
is $\otimes^2$-Sidon.
See Remark \ref{Rf} for a further discussion of possible generalization
of the ``subgaussian" property.

There is an extensive literature on Sidon sets
in commutative discrete groups
or in duals (dual objects)  of compact groups, see e.g. \cite{GH},
but   not much 
   seems to be available on  Sidon sets
 in non-abelian discrete groups or a fortiori in $C^*$-algebras.
 Bo\. zejko  and Picardello investigated several  closely connected notions of Sidon set,
 those that span $\ell_1$ isomorphically but only as a Banach space
 and not an operator space, see \cite{Boz1,Boz2,   Pic}. 
 Apparently no version of Drury's theorem is known for these notions in non-abelian groups.
We refer to Bo\. zejko  and Speicher's \cite{BSp} and also the recent work \cite{BGM}
for some results on completely positive functions on  Coxeter groups that may be related to our own.  See also  \cite{Boz4}. \\See  also \cite{Wang} for a study of Sidon sets in compact quantum groups.\\

 We refer to \cite{P4} for background on completely bounded (c.b. in short),
  completely positive (c.p. in short), and decomposable maps. 
  See also \cite{Haa}. Some of the connections 
  of the latter notions with  the harmonic analysis of the present paper are described
    in chapter 8 and \S 9.6 and \S 9.7 in \cite{P4}. 
 
 \section{Completely Sidon sets}
  Let $U_n$ be the unitary generators in $C^*(\F_\infty)$.\\
 Let $A$ be a $C^*$-algebra.
 Let $(\psi_n)$ be a bounded sequence in $A$.
  \begin{dfn} We say that $(\psi_n)$ is completely Sidon if
 there is $C$ such that
 for any matricial coefficients $(a_n)$
 $$\|\sum a_n \otimes U_n\|_{\min}\le C \|\sum a_n \otimes \psi_n\|_{\min}.$$
 Equivalently, the operator space spanned by $(\psi_n)$  in $A$
 is completely isomorphic to $\ell_1$ equipped with the maximal operator space structure.\\
 Fix an integer $k\ge 1$.
 We say that $(\psi_n)$ is completely $\otimes_{\max}^k$-Sidon
 in $A$
 if the sequence $(\psi_n\otimes \cdots\otimes \psi_n)$  ($k$-times)
 is completely Sidon in $A \otimes_{\max} \cdots\otimes_{\max} A$ ($k$-times).
 \end{dfn}
 \begin{rem} It is important to note that for $k>1$
  the notion of completely $\otimes_{\max}^k$-Sidon 
  is relative to the ambient $C^*$-algebra $A$.
  If $A$ is a $C^*$-subalgebra  of a  $C^*$-algebra $B$,
  and $(\psi_n)$ is completely $\otimes_{\max}^k$-Sidon
 in $A$, it does not follow in general that
 $(\psi_n)$ is completely $\otimes_{\max}^k$-Sidon
 in $B$. This does hold nevertheless if there is 
 a c.p. or decomposable projection from $B$ to $A$.
 It obviously holds without restriction if $k=1$,
 but for $k>1$ the precision $\otimes_{\max}^k$-Sidon
 ``in $A$" is important.
 However, when there is no risk of confusion
 we will omit ``in $A$".
\end{rem}

 \begin{pro}\label{R3} The following are equivalent:
 \item{(i)} The sequence $(\psi_n)$ is completely Sidon.
 \item{(ii)} There is $C$ such that for any ${K},m,k$,
  any   $(a_n)$ in $M_m$, and 
 any    $(u_n)$   in $U(M_k)$ we have
 $$\|\sum\nl_1^{K} a_n\otimes u_n\| \le C \|\sum\nl_1^{K} a_n\otimes \psi_n\|.$$
  \item{(ii)'} Same as (ii) but for   $(u_n)$   in the unit ball of $M_k$.
   \item{(ii)''} There is $C$ such that for any 
   $C^*$-algebras $B$ and $D$,  
  any   $(a_n)$ in $B$, and 
 any    $(u_n)$   in the unit ball of $D$ we have
 $$\|\sum\nl_1^{K} a_n\otimes u_n\|_{B\otimes_{\min} D} \le C \|\sum\nl_1^{K} a_n\otimes \psi_n\|_{B\otimes_{\min} A}.$$
 \item{(iii)} Same as (ii) but with ${K}$ even, say ${K}=2m$ and the $u_n$'s restricted
 to be such that $u_{m+j}=u_j^{-1}$ for $1\le j\le m$.
 \item{(iv)}  Same as (ii) but with  the $u_n$'s restricted
 to be selfadjoint unitaries.
 
 \end{pro}
 \begin{proof}[Sketch]
 The equivalence (i) $\Leftrightarrow$ (ii) is just an  explicit reformulation of the preceding definition.
 To justify (iii) $\Rightarrow$ (ii) we can use $(z u_j, \bar z u_{m+j}^{-1})_{1\le j\le m}$. Then after integrating in $z\in \T$, we can separate the two 
 parts of the sum appearing in (ii). 
 This gives us for the sup over all the $u_n$'s as in (iii)
  $$ {\sup\nl_{(iii)}\|\sum\nl_1^{2m} a_n\otimes u_n\|} \ge 
  \max\{ \|\sum\nl_1^{m} a_n\otimes u_n\|,  \|\sum\nl_{m+1}^{2m} a_n\otimes u_n\| \} .$$
and hence (triangle inequality)
   \begin{equation}\label{e89} {\sup\nl_{(iii)}\|\sum\nl_1^{2m} a_n\otimes u_n\|} \ge 
 (1/2) \sup\nl_{(ii)}\|\sum\nl_1^{2m} a_n\otimes u_n\| ,\end{equation}
 where  the last sup runs over all $(u_n)$ as in (ii).
 We then deduce
 (ii) from (iii) possibly with a different constant.\\
   To justify (iv) $\Rightarrow$ (ii) we can use
   a $2\times 2$-matrix trick: if $(u_n)$ is an arbitrary
   sequence  in $U(k)$, $(\begin{matrix}0 &u_n\\  u_n^* &0  \end{matrix})$
 are selfadjoint in $U(2k)$.  We then deduce
 (ii) from (iv)  with  the same constant.\\
 Lastly the equivalence  (ii) $\Leftrightarrow$ (ii)' 
 is obvious by an extreme point argument,
 and (ii)' $\Leftrightarrow$ (ii)'' 
 (which reduces to $B=B(H)$ and hence to the matricial case)
 follows by  Russo-Dye  and standard operator space arguments
 (see \cite[p. 155-156]{P4}  for more background). 
 \end{proof}
 \begin{rem} For simplicity we state our results
 for sequences indexed by $\N$, but actually they hold with obvious adaptation of the proofs  
 for families indexed by an arbitrary set, finite or not, with bounds independent
 of the number of elements.
 \end{rem}
\n{\bf Examples :} \\
(i) The fundamental example of a completely Sidon set (with $C=1$) is of course
any free subset in a group. 
More precisely, if we add the unit to a free set, the resulting augmented set
is still completely Sidon with $C=1$.
Moreover, any left or right translate of a
completely Sidon set is completely Sidon (with the same $C$).
Therefore any left translate of a free set augmented by the unit
is completely Sidon with $C=1$.
The converse also holds and  is easy to prove, see \cite{Pifz}.\\
(ii) It is proved in \cite[Th. 8.2 p.150]{P4} that for any $G$
the diagonal mapping $t\mapsto \lambda_G(t) \otimes \lambda_G(t) $
defines an isometric embedding of $C^*(G)$ into
 $C_\lambda^*(G)\otimes_{\max} C_\lambda^*(G)$.
 It follows that a subset $\Lambda \subset G$
 is completely Sidon iff the set
 $\{\lambda_G(t)   \mid t\in   \Lambda\}$  is completely $\otimes_{\max}^2$-Sidon   in $C_\lambda^*(G)$.
 Let $M_G$ be the von Neumann algebra of $G$ (i.e. the one generated by
 $   \lambda_G$). 
 Similar arguments show that the same diagonal embedding 
 embeds  $C ^*(G)$ also into $M_G\otimes_{\max} M_G$.
In particular  the set of free generators
is a completely $\otimes_{\max}^2$-Sidon set in $C_\lambda^*(\F_\infty)$
(and also in   $M_{\F_\infty}$).

 \medskip
 
 We will be interested in another property, namely the following one:
 
 Let $X_1,X_2$ be  preduals of  $C^*$-algebras (so-called non-commutative $L_1$-spaces).
 
 We say that a bounded linear map $v: X_2 \to X_1$ is 
 completely positive (in short c.p.)  if $v^*: X_1^* \to X_2^*$ is c.p..\\

   Let $A,B$ be $C^*$-algebras. 
   Let  $CP(A,B)$ be the set of c.p. maps from $A$ to $B$.
   We say that a bounded linear map $u:A\to B$
   is decomposable if 
there are $u_j \in CP(A,B)\quad(j=1,2,3,4)$ such that
$$u=u_1 -u_2 +i(u_3 -u_4 ).$$
 We use the dec-norm as defined by Haagerup \cite{Haa}.
 We denote
 \begin{equation}\label{d11}\| u\|_{dec}=\inf\{\max\{\| S_1\|,\| S_2\|\}\}
 \end{equation}
  where the infimum runs over all maps $S_1 ,S_2\in CP(A,B)$ such that the map
 \begin{equation}\label{d12} V: x\to \left(
\begin{matrix}
 S_1 (x) & u(x)\\
 u(x^* )^* & S_2 (x)
\end{matrix}
\right) \end{equation}
is in $CP(A,M_2 (B))$.  
   
 A  mapping $v: X_2 \to X_1$ is said to be decomposable 
 if its adjoint $v^*: X_1^* \to X_2^*$ is decomposable in the preceding sense
 (linear combination of c.p. maps),
 and we set by convention
 $$\|v\|_{dec} =\|v^*\|_{dec}.$$
We use the term $c$-decomposable
 for maps that are decomposable with dec-norm $\le c$.

 The crucial property of a decomposable map
 $v: A\to B$ between
 $C^*$-algebras is that for any other  $C^*$-algebra
 $C$ the mapping $id_C\otimes v$ extends to a bounded (actually decomposable) map from $C\otimes_{\max} A$
 to 
 $C\otimes_{\max} B$. Moreover we have
 $$  \| id_C\otimes v:   C\otimes_{\max} A\to C\otimes_{\max} B \|
 \le \|v\|_{dec}.
 $$
 Consequently, for any pair 
 $v_j: A_j\to B_j$ ($j=1,2$)
 of decomposable maps 
 between
 $C^*$-algebras, we have
  \begin{equation}\label{e8}
  \| v_1\otimes v_2:   A_1\otimes_{\max} A_2\to B_1\otimes_{\max} B_2\|_{cb} \le
  \| v_1\otimes v_2:   A_1\otimes_{\max} A_2\to B_1\otimes_{\max} B_2\|_{dec}  \end{equation}
  \begin{equation}\label{e8'}
\le \|v_1\|_{dec}\|v_2\|_{dec}.
 \end{equation}
 
 \begin{dfn}\label{dom} Let $I,J$ be any sets.
  (i) Let $(x_n^1)_{n\in I } $ (resp.    $(x_n^2)_{n\in I } $ )
  be a family in $X_1$  (resp. $X_2$).
  Let us say that $(x_n^1)_{n\in I } $ is $c$-dominated
  (or ``decomposably $c$-dominated")
  by $(x_n^2)_{n\in I} $ if there is   a decomposable mapping 
  $v: X_2 \to X_1$ 
  with $\|v\|_{dec}\le c$ such that $v(  x_n^2 ) = x_{n}^1 $
  for any $n\in I$.\\
   We simply say ``dominated" for  $c$-dominated for some $c$.\\
  (ii) We say that $(x_n^1)_{n\in I } $ and  $(x_n^2)_{n\in J } $
  are `` decomposably equivalent "  
  if there is a bijection $f: I \to J$ such that
    each of the families
    $(x_n^1)_{n\in I } $ and $(x_{f(n)}^2)_{n\in I} $  is dominated by the other.
\end{dfn}
Let $Y$ be the predual of a $C^*$-algebra.
 The positive cone in $M_k(Y)$ is the polar of the 
 positive cone $M_k(Y^*)_+$  in the $C^*$-algebra $M_k(Y^*)$.
 More precisely
 $y\in M_k(Y)_+$
 iff  
 $$\forall a\in M_k(Y^*)_+\quad \sum\nl_{ij}  a_{ij} (y_{ij} )\ge 0. $$
 
 Clearly $v: X_2 \to X_1$ is c.p.
 iff for any $k$ the mapping $id_{M_k} \otimes v: M_k(X_2)
 \to M_k(X_1)$ is positivity preserving.

 More generally,   since we have positive cones  on 
 both $M_k(X^*)$ and   $M_k(Y)$,
 we can extend the definition of complete positivity
 to maps from a $C^*$-algebra to $Y$
 or from $Y$ to a $C^*$-algebra. In particular,
  a map $T: X^*\to Y$
is called c.p. if  
$id_{M_k} \otimes T: M_k(X^*)\to M_k(Y)$ is positivity preserving
for any $k$.
\begin{rem}\label{rop}[Opposite von Neumann algebra]
 The opposite von Neumann ${M^{op}}$ is the same linear space
as $M$ but with the reverse product.
Let $\Phi: {M^{op}} \to M$ be the identity map,
viewed as acting from ${M^{op}}$ to $ M$,
so that  $\Phi^*: M^* \to {M^{op}}^*$ also acts as the identity.\\
  When $M$ is a von Neumann algebra
equipped with a normal faithful tracial state $\tau$,
there is a minor problem that needs clarification.
We have a natural inclusion $J: M\to M^*$ denoted by $y\mapsto J y$ and
defined by $J{y}(x)=\tau(yx)$.
In general this is not c.p, but
it is c.p. when viewed as a mapping
either from $M^{op}\to M^*$ or from $M\to {M^{op}}^*$.
Indeed,  for all $x,y\in M_k(M)_+$ 
we have $[\tr\otimes\tau] (xy)=\sum_{ij}  \tau ( x_{ij} y_{ji}) \ge 0$
but in general it is \emph{not  true} that for $\sum_{ij}  \tau ( x_{ij} y_{ij})=
\sum_{ij}  J{y_{ij}} ( x_{ij} )  $.\\
 Then the content of the preceding observation is that
$\Phi^* J: M \to {M^{op}}^*$ is c.p. (but $J$ in general fails this). 
  \end{rem}

    \begin{rem}\label{R6}[About preduals of finite vN algebras]
 Let  $(M^1,\tau^1)$ be here any noncommutative probability space,
 i.e. a von Neumann algebra equipped with
 a normal faithful tracial state.
 The predual ${M^1}_*$  is the subset of ${M^1}^*$ formed of the
 weak* continuous functionals  on $M^1$. It can be
isometrically identified with the space
 $L_1(\tau_1)$ defined as the completion of $M_1$ for the 
 norm $\|x\|_1=\tau_1(|x|)$. Thus 
 we have a natural inclusion
 with dense range  $M^1\subset L_1(\tau_1)$.
 We need to observe
 the following fact. Let $({M}^2, \tau^2)$
 be another noncommutative probability space.
 Let $V: L_1(\tau_1)\to L_1(\tau_2)$
 be a linear map that is a $*$-homomorphism
 from $M^1$ to $M^2$
 when restricted to $M^1$.
 Then $V$ is 
 completely positive and hence 1-decomposable.
 \end{rem}
  \section{Analysis of the free group case}\label{sf0}

 We denote by $M$ the von Neumann algebra of the free group
 $\F_{\infty}$  equipped
 with its usual trace $\tau$.
 
We denote by $(\y_n)_{ n\ge 1}$   the elements of $M=\lambda_{\F_{\infty}} (\F_{\infty})''$
corresponding
 to the   free generators $(g_n)$ in $\F_{\infty}$,
 i.e. $\y_n=\lambda_{\F_{\infty}}(g_n)$.
 
 For convenience we set  
 $$\forall n\ge 1\quad  \y_{-n}=\y_n^{-1}.$$  
Although this is a bit pedantic, it is wise
to distinguish the elements of $M$ from the linear functionals 
on $M$ that they determine. Thus we let 
 $(y_n)_{n\in \Z_*}$ be the sequence in $M_*\subset M^*$
 that is biorthogonal to the sequence $({\varphi}_n)_{n\in \Z_*}$, and defined
 for all $n\in \Z_*=\Z\setminus \{0\}$ by
 \begin{equation}\label{e11}\forall a\in M\quad   {y_n}(a)=\tau (\varphi_n^* a).\end{equation}
 We also define $y_n^*\in M_*\subset M^*$ as follows
  \begin{equation}\label{e12}\forall a\in M\quad   {y^*_n}(a)=\tau (\varphi_n a).\end{equation}
Again let $J: M \to M^*$ be the inclusion
 mapping defined by 
 $Ja(b)=\tau(ab)$. With this notation
 $$y_n=J(\varphi_n^*) \text{  and  } y_n^*=J(\varphi_n) .$$
 
 For future reference, we record here a simple observation:
  \begin{lem}\label{dom1}
  Recall $\Z_*=\Z\setminus \{0\}$. The families 
  $(y_n)_{ n\ge 1}$ and $(y_n)_{ n\in \Z_*}$
  are  decomposably equivalent  in the sense of Definition \ref{dom}.
\end{lem}
   \begin{proof} Let $(z_n)_{n\in \Z_*}$ be a sequence
   such that each $(z_n)_{n>0}$ and $(z_n)_{n<0}$
   are mutually  free, each one being a free Haar unitary sequence.
   Then $(z_n)_{n\in \Z_*}$ and $(\y_n)_{ n\ge 1}$ are  trivially
    decomposably equivalent.
   Let $L$ be a copy of the von Neumann algebra of $\Z$.
   Let $N=M\ast L$. Let $U$ denote the generator of $L$
   viewed as a subalgebra of $N$.
   We also view $M\subset N$. 
    Then the family $(Uy_n)_{ n\in \Z_*}$
    viewed as sitting in $N_*$
    is a family of free Haar unitaries.
    Therefore $(Uy_n)_{ n\in \Z_*}$ and 
    $(y_n)_{ n\ge 1}$ are   
    decomposably equivalent. But $(Uy_n)_{ n\in \Z_*}$
   and $(y_n)_{ n\in \Z_*}$
are also  decomposably equivalent  in $N_*$, because the multiplication
by $U$ or $U^{-1}$ is    decomposable
(roughly because, since $x\mapsto a xa^*$ is c.p.,
$x\mapsto a xb^*$ is decomposable by the polarization  formula).
Lastly using conditional expectations
it is easy to see that the  families
$(y_n)_{ n\in \Z_*}\subset N_* $
and $(y_n)_{ n\in \Z_*}\subset M_* $ (identical families
viewed as sitting in $N_* $ or $M_* $)
are  decomposably equivalent.
   \end{proof}
 Let $\cl A$ be the algebra generated by $(\varphi_n)_{n\in \Z_*}$.
 Note for further reference
 that the orthogonal projection $P_1$
 onto the closed span in $L_2(\tau)$ of $(\varphi_n)_{n\in \Z_*}$ 
 is defined by
 $$\forall a\in \cl A\quad  P_1(a)=\sum \tau(\varphi_n^*a) \varphi_n. $$

 We use ingredients analogous to those of \cite{Pi3}
 but in \cite{Pi3} the free group is replaced by the free Abelian group,
 and an ordinary gaussian sequence is used
 (we could probably use analogously a free semicircular sequence here).
 
 Let $U_n\in C^*(\F_\infty)$  be the unitaries coming from the  
 free generators. We set again by convention 
 $U_{-n}=U_n^{-1}$ ($  n\ge 1$).\\
 Let $\cl E={\rm span}[U_n\mid n\in \Z_*] \subset C^*(\F_\infty)$.
 Consider the
  natural 
  linear map
   $\pi:  \cl E\to M$
  such that
  $$\forall n\in \Z_*\quad\pi(U_n)={\varphi}_n.$$ 
  Its key property is that
  for some Hilbert space $H$ there is  a factorization
  of the form
  $$\cl E {\buildrel \pi_1\over
\longrightarrow } B(H) {\buildrel \pi_2\over
\longrightarrow }  M$$
such that
$$\forall n\in   {\Z_*}\quad \pi_2\pi_1(U_n)=\pi(U_n)={\varphi}_n$$
where $\|\pi_1\|_{cb}\le 1$, and $\pi_2$ is a
\emph{decomposable} map
with
$  \|\pi_2\|_{dec}=1$.
To check this note that $M$ embeds in a trace preserving way into
an ultraproduct $\cl M$ of matrix algebras, and there is a c.p. conditional
expectation from $\cl M$ onto $M$.
Therefore  
 there is a completely positive surjection $\pi_2$
 from $B=\prod\nl_{k} M_k$ to $M$ and a  $*$-homomorphism
 $\pi_1: C^*(\F_\infty) \to B$ such that ${\pi_2 \pi_1}_{| {\cl E }}  =\pi$.
To complete the argument we need to replace $B$ by $B(H)$. Since
$B$ embeds in $B(H)$ for some $H$
and there is a conditional expectation from 
 $B(H)$  to $B$, this is immediate.
We refer the reader to \cite[\S 9.10]{P4}   for more details.

 The following   statement
 on the free group factor $M$
 is the key for our results.
 \begin{thm}\label{key4}
  The sequence $(\y_n)_{ n\in {\Z_*}}$ in $   M$
 satisfies the following property:\\
 any bounded sequence $(z_n)_{ n\in {\Z_*}}$ in $M$ that is 
 biorthogonal to $(\y_n)_{ n\in {\Z_*}}$ 
 in $L_2(\tau)$ meaning that 
 $$\tau(z_n \y_m^*)=0 \text{  if  } n\not = m \text{  and  } \tau(z_n \y_n^*)=1,$$
 is completely $\otimes^2_{\max}$-Sidon.
 More generally, if $(z^1_n)_{ n\in {\Z_*}}$
 and $(z^2_n)_{ n\in {\Z_*}}$ are   bounded in $M$ and each biorthogonal to $(y_n)_{ n\in {\Z_*}}$,
 then $(z^1_n\otimes z^2_n)_{ n\in {\Z_*}}$ is completely Sidon in $M \otimes_{\max} M$.
  \end{thm}
    
   Let $(z^1_n)$
 and $(z^2_n)$ be as in Theorem \ref{key4}.
   Assume $\|z^j_n\|\le C'_j$ for all $n\in {\Z_*}$ ($j=1,2$).
 
 Fix integers $k,k'\ge 1$. Let $(a_n)$ be a family in $M_k$ with only finitely many $n$'s
 for which $a_n\not= 0$.
 Let  $(u_n)_{n\in {\Z_*}}$ be unitaries in $M_k$ such that 
 $u_{-n}=u_n^{-1}$ for all $n$.  
Our goal is to  show that there is a constant $\alpha$ depending only 
on  $C'_1,C'_2$ such that
$$ \|\sum u_n \otimes a_n \|_{M_{k'}\otimes_{\min}  M_k}\le \alpha\|  \sum a_n \otimes z^1_n \otimes  {z^2_n}^* \|_{M_k(     M \otimes_{\max} M^{op})}.$$
This will prove the key Theorem \ref{key4}
with $M \otimes_{\max} M^{op}$ instead of 
$M \otimes_{\max} M$. Then a simple elementary argument will allow us
to replace $M^{op}$ by $M$.

 \begin{rem}\label{r1}
Let $T: M \to M^*$ be  a  c.p. map such that $T(1)(1)=1$.
We associate to it a state $f$ on $M\otimes_{\max}  M$
by setting
$$f(x\otimes y)= T(x)(y).$$
A matrix $x\in M_k (M^*)$
is defined as $\ge 0$ if $\sum_{ij} x_{ij}(y_{ij}) \ge 0$
for all $y\in M_k(M)_+$.\\
More generally, any decomposable
operator $T$ on $M$ (in particular any finite rank one)
determines an element $\Phi_{T}$ of $(M\otimes_{\max} M^{op})^*$,
defined by for $x,y\in M$ by
$$\langle \Phi_{T}, x\otimes y\rangle=
\tau(T(x)y).
$$
Indeed, the bilinear form
$(x,y)\mapsto \tau(xy)$ is of unit norm in 
$(M\otimes_{\max} M^{op})^*$ and
 $$\|T\otimes id_{M^{op}}: M\otimes_{\max} M^{op}\to M\otimes_{\max} M^{op} \|\le \|T\|_{dec}.$$
Furthermore, for any pair of $C^*$-algebras $A,B$, we  have a 1-1-correspondence between
the set of decomposable maps $T: A\to B^*$ and 
$(A\otimes_{\max} B)^*$.

\end{rem} 
 \begin{rem}\label{r2}
 We will need the free analogue of Riesz products.\\
Recall  we set $M=M_{\F_\infty}$.
Let $0\le \vp\le 1$.
Let $P_\ell$ the orthogonal projection on $L_2(\tau)$ onto the span of the words
of length $\ell$ in $\F_\infty$. Let ${\theta}_\vp =\sum_{\ell\ge 0} \vp^\ell P_\ell$.
  By Haagerup's well known result \cite{Hainv},
${\theta}_\vp$ is a c.p. map on $M$. Composing it with the inclusion
$M\subset M_*$, we find a unital c.p. map
from $M  $ to $ {M^{op}}^*$, and hence ${\theta}_\vp$ 
determines a state $f_\vp$ on $M\otimes_{\max} M^{op}$.
 
 We view $\theta_\vp$ as acting
 from $M$ to $L_2(\tau)$.
We can also consider it as a map taking $\cl A$ to itself.

We  will crucially use the   decomposition
$  ({\theta}_\vp-{\theta}_0)/\vp= P_1+  \sum_{{\ell}\ge 2} \vp^{{\ell}-1} P_{\ell}.$
We set $$T_\vp=({\theta}_\vp-{\theta}_0)/\vp\text{ and }
  R_\vp= - \sum_{{\ell}\ge 2} \vp^{{\ell}-1} P_{\ell},$$
  so that
 \begin{equation}\label{e0} P_1= T_\vp+R_\vp.\end{equation}
We have
 \begin{equation}\label{e1}\|T_\vp\|_{dec}\le 2/\vp\end{equation}
and
 \begin{equation}\label{e2}\|R_\vp: M \to L_2(\tau)\|\le \|R_\vp: L_2(\tau) \to L_2(\tau)\|\le  \vp.\end{equation}
\end{rem}

\begin{lem}\label{l1}   
  With the preceding notation,   we have
 \begin{equation}\label{e3}   \|  \sum a_n \otimes T_\vp(z^1_n) \otimes  {z^2_n}^* \|_{M_k(     M \otimes_{\max} M^{op})}
 \le (2/\vp) \|  \sum a_n \otimes z^1_n \otimes  {z^2_n}^* \|_{M_k(     M \otimes_{\max} M^{op})}
   .\end{equation}
\end{lem}
 \begin{proof}
 This follows from \eqref{e8'}. \end{proof}
 \begin{proof}[Proof of Theorem \ref{key4}]
 Fix $\vp<1$ (to be determined later). We have   decompositions
 $$T_\vp(z^1_n) = \y_n + r^1_n $$
 $$ z^2_n = \y_n + r^2_n $$
 where $r^1_n= -R_\vp(z^1_n  )$ and $r^2_n $
 are orthogonal to  $(\y_n)_{n\in {\Z_*}}$ and moreover
 $$\|r^1_n\|_2=\|R_\vp(z^1_n  )\|_2 \le \vp \|z^1_n\|_2    \le \vp    C'_1,$$
 $$\|r^2_n\|_2\le \| z^2_n\|_2 \le     C'_2.$$
 
 We have
 $$T_\vp(z^1_n) \otimes  {z^2_n}^*
 =(\y_n + r^1_n) \otimes (\y_n + r^2_n)^*.$$
 The idea will be to reduce this product to the simplest term $\y_n  \otimes \y_n^*$.\\
 Let $V: M\to M_k(M)$ be the isometric $*$-homomorphism
 taking $\varphi_n$ to $u_n\otimes \varphi_n$.
 Note that  $V$ is     decomposable with $\|V\|_{dec}=1$.
 We observe
 $$(V\otimes  id_{M^{op}} )(\y_n   \otimes \y_n   ^* ) =u_n \otimes \y_n   \otimes \y_n^*.$$
 Let $\gamma: M\otimes M^{op} \to \C$ be the bilinear
form defined by $\gamma(a\otimes a')=\tau(aa')$.
It is a classical fact that $\gamma$ is a state on $M\otimes_{\max} M^{op} $.
 We claim
  \begin{equation}\label{e88}\|(id_{M_k} \otimes \gamma)(V\otimes  id_{M^{op}} )(r^1_n   \otimes {r^2_n}^* ) \|_{M_k}
 \le \vp C'_1C'_2.\end{equation}
Indeed, let $\F=\F_\infty$ for simplicity.
We may develop in $L_2(\tau)$
$$r^1_n= \sum\nl_{t\in \F}  r^1_n(t) \lambda_\F(t) \text{  and   }
 r^2_n= \sum\nl_{t\in \F}  r^2_n(t) \lambda_\F(t).$$
 Let $\sigma$ be the unitary representation
 on $\F$ taking $g_n$ to $u_n\in M_{k'}$. For simplicity
 we denote $u_t=\sigma(t)$ for any $t\in \F$. With this notation
 $V(\lambda_\F(t))= u_t\otimes \lambda_\F(t).$
Then
$$(V\otimes  id_{M^{op}} )(r^1_n)=  \sum\nl_{t\in \F}  r^1_n(t) u_t\otimes \lambda_\F(t),$$
$$(id_{M_{k'}} \otimes \gamma)(V\otimes  id_{M^{op}} )(r^1_n   \otimes {r^2_n}^* )=
 \sum\nl_{t\in \F}  r^1_n(t) \ovl{ r^2_n(t)  } u_t $$
and hence (triangle inequality and Cauchy-Schwarz)
$$\|(id_{M_{k'}} \otimes \gamma)(V\otimes  id_{M^{op}} )(r^1_n   \otimes {r^2_n}^* ) \|_{M_{k'}}
 \le\|r^1_n\|_2 \|r^2_n\|_2 \le \vp C'_1C'_2 .
 $$
 This proves our claim.
 Let
 \begin{equation}\label{e20}\d_n= (id_{M_{k'}} \otimes \gamma)(V\otimes  id_{M^{op}} )(r^1_n   \otimes {r^2_n}^* ).\end{equation}
 Recalling the orthogonality relations $\y_n \perp r^1_n$ and $\y_n \perp r^2_n$
 we see that
 $$(id_{M_{k'}} \otimes \gamma)(V\otimes  id_{M^{op}} ) (  T_\vp(z^1_n)   \otimes {z^2_n}^* )
)
 =  (id_{M_{k'}} \otimes \gamma)(V\otimes  id_{M^{op}} )(\y_n   \otimes {\y_n}^* )
 +(id_{M_{k'}} \otimes \gamma)(V\otimes  id_{M^{op}} )(r^1_n   \otimes {r^2_n}^* )$$
 $$= u_n  + \d_n.$$
 We now go back to \eqref{e3}:
 we have
 $$
  (id_{M_k} \otimes id_{M_{k'}} \otimes \gamma)(V\otimes  id_{M^{op}} )\sum a_n \otimes T_\vp(z^1_n) \otimes  {z^2_n}^*  =
   \sum a_n \otimes (u_n+ \d_n).  $$
   Therefore (the norm $\|\ \|$ is the norm in ${M_{k'}\otimes_{\min}  M_k}$)
   $$\|  \sum a_n \otimes (u_n+ \d_n)\|\le \|   \sum a_n \otimes T_\vp(z^1_n) \otimes  {z^2_n}^*   \|_{M_k(     M \otimes_{\max} M^{op})}
$$
and hence by \eqref{e3}
$$\|  \sum a_n \otimes (u_n+ \d_n)\|\le (2/\vp) \|   \sum a_n \otimes z^1_n  \otimes  {z^2_n}^*   \|_{M_k(     M \otimes_{\max} M^{op})}.
$$
By the triangle inequality
$$\|  \sum a_n \otimes u_n \|  -\|  \sum a_n \otimes \d_n \|   \le (2/\vp) \|   \sum a_n \otimes z^1_n  \otimes  {z^2_n}^*   \|_{M_k(     M \otimes_{\max} M^{op})} .
$$
Recalling \eqref{e88} and \eqref{e20}
we find
$$\|  \sum a_n \otimes u_n \| -\vp C'_1C'_2 \sup\nl_{b_n\in B_{M_{k'}}}\|  \sum a_n \otimes b_n \|\le (2/\vp) \|   \sum a_n \otimes z^1_n  \otimes  {z^2_n}^*   \|_{M_k(     M \otimes_{\max} M^{op})} .
$$
Taking the sup over all $u_n$'s and using \eqref{e89} (recall $B_{M_{k'}}$ is the convex hull
of $U(k')$)
we find
$$(1/2-\vp C'_1C'_2)  \sup\nl_{b_n\in B_{M_{k'}}}\|  \sum a_n \otimes b_n \|\le (2/\vp) \|   \sum a_n \otimes z^1_n  \otimes  {z^2_n}^*   \|_{M_k(     M \otimes_{\max} M^{op})} .
$$
 This completes the proof for $M \otimes_{\max} M^{op}$, since
 if we choose, say, $\vp=\vp_0$ with $\vp_0=(4 C'_1C'_2)^{-1}$
 we obtain the announced result with $\alpha= 8/\vp_0=32 C'_1C'_2$.

It remains to justify the replacement of $M^{op}$ by $M$.
For this it suffices to exhibit
  a (normal) $\C$-linear $*$-isomorphism $\chi: M^{op}\to M$
such that   $\chi(\varphi^*_n)= \varphi_n$ for all $n\in {\Z_*}$.
Indeed, let us view $M\subset B(H)$
 with $H=\ell_2(\F_\infty)$. 
 Then
 since ${}^t\varphi_n=\varphi_n^*$ for all $n\in {\Z_*}$
 (these are matrices with real entries),  the matrix transposition $ x\mapsto {}^t x$
 is the required $*$-isomorphism $\chi: M^{op}\to M$.
  \end{proof}
    
   \begin{cor}\label{key2} Let $(z^1_n)_{ n\ge 1}$
 and $(z^2_n)_{ n\ge 1}$ be    bounded in $M$ and each biorthogonal to $(y_n)_{ n\ge 1}$,
 then $(z^1_n\otimes z^2_n)_{ n\ge 1}$ is completely Sidon in $M \otimes_{\max} M$.
   \end{cor}
    \begin{proof} By Lemma \ref{dom1} we know that
     $(y_n)_{ n\ge 1}$ is dominated by $(y_n)_{ n\in {\Z_*}}$.
     Let $v: M_* \to M_*$ decomposable taking $(y_n)_{ n\in {\Z_*}}$
     to $(y_n)_{ n\ge 1}$ (modulo a suitable bijection $f: {\Z_*}\to \N_*$).
     Then $(v^*(z^j_{f(n)}))_{ n\in {\Z_*}}$ ($j=1,2$)
     is   biorthogonal to $(y_{f(n)})_{ n\in {\Z_*}}$.
     By Theorem \ref{key4},  
       $(v^*(z^1_{f(n)}) \otimes v^*(z^2_{f(n)}))_{ n\in {\Z_*}}$
     is completely Sidon in $M\otimes_{\max} M$.
     By \eqref{e8'}
     $( z^1_{f(n)} \otimes  z^2_{f(n)} )_{ n\in {\Z_*}}$
     is completely Sidon in $M\otimes_{\max} M$.
     Equivalently since this   is obviously invariant under permutation,
     we conclude
      $( z^1_{n} \otimes  z^2_{n} )_{ n\in \N_*}$
     is completely Sidon in $M\otimes_{\max} M$.
     
   \end{proof}
   
   \section{Main results. Free unitary domination}\label{sf}

 We start with a simple but crucial observation that links
 completely Sidon sets with 
 the free analogues of Rademacher functions or independent gaussian
 random variables.
 \begin{pro}\label{p1}
 Let $\Lambda=\{\psi_n\mid n\ge 1\}$  
be a completely Sidon set in $A$ with constant $C$.
Then there is a biorthogonal system
$(x_n)_{n\ge 1}$ in $A^*$
that is $C$-dominated
by   $(y_n)_{n\ge 1}$. 
 \end{pro}
  \begin{proof}  
Let $E\subset A$ be the linear span
of $\{\psi_n\} $.
Let $\alpha: E \to  \cl E $
be the linear map such that $\alpha (\psi_n)= U_n$.
By our assumption $\|\alpha\|_{cb}\le C$.
We have $\|\pi_1 \alpha: E\to B(H)\|_{cb}\le C$.
By the injectivity of $B(H)$,
$\pi_1 \alpha$ admits  an extension  
$\beta: A \to B(H)$ with $\|\beta\|_{cb}\le C$.
Note (see \cite{Haa}) that 
$\|\beta\|_{dec}=\|\beta\|_{cb}$.
Let $V=\pi_2\beta: A \to M$.
Then $V$ is a $C$-decomposable  map.
Its adjoint $V_*: M_* \to A^*$ is also $C$-decomposable.
Let $ {y_n}\in M_*$ be the functionals
biorthogonal to  the sequence $({\varphi}_n)$ 
  defined above in $M$.
We have $ {y_n}({\varphi}_m)=\d_{nm}$.
Therefore since $V(\psi_m) =\pi_2\beta(\psi_m)
=\pi_2\pi_1 \alpha(\psi_m)={\varphi}_m$
$$  {y_n} (V(\psi_m))=\d_{nm}.$$
Thus setting $x_n=V_*( {y_n})$
we find  $   {x_n} (\psi_m)=\d_{nm}.$
This shows that $(x_n)$,
which is by definition $C$-dominated
by $(y_n)$,  is biorthogonal to $(\psi_n)$.
  \end{proof}
 \begin{thm}\label{t1} 
 Let   $A_1$, $A_2$ be $C^*$-algebras.
 Let $(\psi^1_n)_{n\ge 1}$, $(\psi^2_n)_{n\ge 1}$  be bounded sequences  in $A_1$, $A_2$
 bounded by $C_1'$ and  $C_2'$ respectively. Let
   $(x^1_n)_{n\ge 1}$ be a sequence  in $A_1^*$ biorthogonal to $(\psi^1_n)_{n\ge 1}$,
   and let $(x^2_n)_{n\ge 1}$ be a sequence  in $A_2^*$ biorthogonal to $(\psi^2_n)_{n\ge 1}$.
   If both are dominated by $(y_n)_{n\ge 1}$, then
   $(\psi^1_n\otimes \psi^2_n)_{n\ge 1}$ is completely Sidon in $A_1\otimes_{\max} A_2$.\\
 More precisely, if   $(x^j_n)_{n\ge 1}$ is $c_j$-dominated by $(y_n)_{n\ge 1}$,
 $(\psi^1_n\otimes \psi^2_n)_{n\ge 1}$ is completely Sidon in $A_1\otimes_{\max} A_2$ with a constant $C$ depending only on $C_1',C_2', c_1,c_2$.
 \end{thm}

 \begin{proof} 
The key ingredient is Corollary  \ref{key2}.
 Assume $(x_n^j)$ dominated by $(y_n)$.
 Let $v_j: M_* \to A^*_j$ be decomposable such that
 $v_j(y_n)=x_n^j$  ($j=1,2$), with $(x_n^j)$
 biorthogonal to  $(\psi_n^j)$ and $\|v_j\|_{dec}\le c_j$.
 Moreover let $w_j: A_j \to M$ be the restriction of 
 $v_j^{*}: A_j^{**} \to M$   to $A_j$.
 Note that $(v_j^{*} (\psi^j_n))$, or equivalently
  $(w_j  (\psi^j_n))$, is obviously biorthogonal
  to $(y_n)$ for each $j=1,2$.
  Let $z_n^j=w_j  (\psi^j_n)$.
 By Corollary  \ref{key2}
 the sequence $(z_n^1 \otimes z_n^2)$
 is completely Sidon in $M \otimes_{\max} M$.
 But since $w_j\in D(A_j,M)$
 we see by \eqref{e8'}
 that this implies that 
 $(\psi^1_n\otimes \psi^2_n)$ is completely Sidon in $A_1\otimes_{\max} A_2$.
 The assertion on the constants is easy to check by going over the argument.
  \end{proof}

  \begin{rem}[On ``pseudo-free" sequences]\label{ps}
Let us say that a sequence $({\kappa}_n)$ in the predual $N_*$ of a von Neumann algebra $N$
is pseudo-free if $(y_n)$ and $({\kappa}_n)$ are  decomposably equivalent. 
Clearly, we may replace $(y_n)$ by any other pseudo-free sequence
in  what precedes.  
Note that   any sequence of free Haar unitaries,
(or  free Rademacher) or of free semicircular variables
is  pseudo-free.
More generally,
any free sequence $({\kappa}_n)$ with mean 0 in a non-commutative tracial
probability space
 such that   $\inf\|{\kappa}_n\|_1>0$  and $\sup\|{\kappa}_n\|_\infty<\infty$
 is pseudo-free.\\
 Indeed, this can be deduced
from the fact that trace preserving unital c.p. maps
extend to  trace preserving c.p. maps on reduced free products.
The latter fact reduces the problem to the commutative case
(one first checks the result for a single variable
with unital c.p. maps instead of decomposable ones).
\end{rem}

  \section{The union problem}
  
  It is high time to formalize a bit more the central notion
  of this paper.  
 \begin{dfn}\label{d1} Let $(x_n) _{n\ge 1}$ 
 be a sequence in the predual $X$ of a von Neumann algebra.
 Let $(y_n)$ be as before in $M_*$.
 We will say that $(x_n)_{n\ge 1} $  
 is free-gaussian dominated  in $X$ (or  dominated by free-gaussians in $X$)
 if  it is (decomposably) dominated by  the sequence
 $(y_n)$ in $M_*$, or equivalently (see Remark \ref{ps})
 if it is (decomposably) dominated by a free-gaussian sequence
 (or any pseudo-free sequence)
 in $M_*$. Here ``(decomposably) dominated" is meant 
  in the sense of Definition \ref{dom}.\\
 For convenience we define the associated constant using the 
 (unitary) sequence $(y_n)$:
we say that $(x_n)_{n\ge 1} $ is free-gaussian $c$-dominated in $X$
  if it is $c$-dominated by $\{y_n\mid n\ge 1\} \subset M_* $, so that we have
 $T: M_*\to X$ with $\|T\|_{dec}\le c$
 such that $T(y_n)=x_n$.
 \end{dfn}  
   \begin{rem} By classical results (see \cite[p. 126]{Tak})
   for any  von Neumann algebra
   ${\cl M}$,
   there is a c.p.   projection
   (with dec-norm equal to 1) from ${\cl M}^*=({\cl M}_*)^{**}$ to ${\cl M}_*$.
   Therefore the notions of domination
   in ${\cl M}_*$  or in ${\cl M}^*$ are equivalent for sequences sitting
   in  ${\cl M}_*$. 
 \end{rem}
 
Of course we frame the preceding definition 
to emphasize the analogy with the sequences dominated by
i.i.d gaussians in \cite{Pi3}. Note that in the latter,
with independence in place of freeness, 
dominated by gaussians does not imply
dominated by  i.i.d. Haar unitaries,
(indeed gaussians themselves fail this)
but  it holds in the free case
because  free-gaussians are bounded.
Note in passing that  bounded linear maps
between  $L_1$-spaces of commutative (and hence injective) von Neumann algebras
are automatically decomposable.

 \begin{lem}\label{L9} Let ${\cl M},{\cl N} $
 be   von Neumann algebras. Assume 
 that  ${\cl N}$ is equipped
  with   a normal  faithful  tracial state $ \tau' $.
   For $a\in {\cl N}$ we denote by $\hat a \in {\cl N}_*$
  the associated linear form on ${\cl N}$ defined by
  $\hat a(x)=\tau'(a^*x)$.   
 Let $(v_n)_{n\ge 1}$ be unitaries in ${\cl N}$, so that
  $\hat v_n\in N_*  $.
  Let  $(x_n)_{n\ge 1}\in {\cl M}_*$ be
 free-gaussian $c$-dominated.
Then the sequence $(x_n\otimes \hat v_n  )\in  ({\cl M}\ovl\otimes {\cl N})_*$
 is also $c$-dominated by $(y_n)$.
\end{lem}
 \begin{proof}
Let $T$ be as in Definition \ref{d1} (here $X=   {\cl M}_*$).
Since $(y_n\otimes \hat v_n)$ and $(y_n )$ have the same
$*$-distribution,  
the linear mapping $W$ taking  $y_n $ to $y_n\otimes \hat v_n$
extends to a  c.p. (isometric, unital and trace preserving) map
$W$ from ${M}_*$ to $({ M}\ovl\otimes {\cl N})_*$ (see Remark \ref{R6}).
Then the composition $T_1=(T \otimes id_{{\cl N}_*}) W$ 
  takes
$y_n$ to $x_n \otimes \hat v_n$.
Since $W^*$ is c.p.
and $ (T \otimes id_{{\cl N}_*})^*=T^*\otimes id_{{\cl N}}$,
with dec-norm $\le c$,
$T_1$ is $c$-decomposable.
\end{proof}
\begin{rem}\label{dye}  By the Russo-Dye theorem
the unit ball of ${\cl N}$ is the closed convex hull of its unitaries.
Actually, for any fixed $0<\d<1$
there is an integer $K_\d$
such that any $v\in {\cl N}$ with norm $<\d$
can be written as an average of $K_\d$ unitaries,
this is due to Kadison and Pedersen, see \cite{KaPe} for a proof with $\d=1-2/n$ and $K_\d=n$.
Using this, we can extend Lemma \ref{L9}
to sequences $(v_n) $ in the unit ball of ${\cl N}$.
Indeed, the set of sequences  $(v_n)$ in ${\cl N}$
such that   the  sequence  $(x_n\otimes \hat v_n )$ in $({\cl M}\ovl\otimes {\cl N})_*$
is $c$-dominated by $(y_n)$ is obviously a convex set.
By Lemma \ref{L9} it contains the set of sequences 
of unitaries in ${\cl N}$.
By the Kadison-Pedersen result
it contains any family $(v_n)$ with
 $\sup\|v_n\|<1$.
Therefore if $\|v_n\|\le 1$ for all $n$  
the sequence $(x_n\otimes \hat v_n )\in  ( {\cl M} \ovl\otimes {\cl N})_*$
 is   $c(1+\vp)$-dominated by $(y_n) $ for any $\vp>0$.
\end{rem}
 
\begin{lem}\label{L8}
Let $\Lambda=\{\psi_n\}$  
be a completely Sidon set in ${\cl M}$ with constant $C$.
    There is a biorthogonal system
$(x_n)$ in ${\cl M}^*$
such that,  for any $(\cl N,\tau')$ as before and any $(v_n)$ with
$\sup\nl_n\|v_n\|<c'$,
the sequence
$(x_n\otimes  \hat v_n )\in  ( {\cl M}^{**} \ovl\otimes {\cl N})_*$
 is    free-gaussian $Ccc'$-dominated.
\end{lem}
 \begin{proof}
  This follows from Proposition \ref{p1}
  and Lemma \ref{L9} with the variant described in Remark \ref{dye}, applied
  to $(v_n/c')$.
\end{proof}
\begin{rem}\label{R9}
At this point it is useful to observe the following:
consider  two sequences in $X$ (a von Neumann algebra predual)
each of which is free-gaussian $C$-dominated,
we claim that their union 
is  free-gaussian $2C$-dominated.
Indeed,  if we have $x_n^j=T_j(y_n)$, $x_n^j\in X$,
with $T_j: M_* \to X$, $\|T_j\|_{dec}\le C$
($j=1,2$).
Let $M \ast M$ be the free product,
and let $E_j$ ($j=1,2$) be the conditional expectation  
onto each copy of $M_*$ in 
$(M \ast M)_*$.
 We can form the operator
$T: (M \ast M)_*\to X$ defined by
$T(a)= T_1E_1 (x)+ T_2E_2 (x)$.
Clearly $T$ is decomposable with $\|T\|_{dec}\le 2C$.
Let $(y^1_n)$ and $(y^2_n)$ denote the sequences
corresponding to $(y_n)$ in each copy of $M_*$ in 
$(M \ast M)_*$. We have
$T(y^j_n)=x^j_n$ for all $n$ and  all $j=1,2$.
But since  the sequence $\{y^1_n\} \cup \{y^2_n\}$
is clearly equivalent to our original sequence $ \{y_n\}$,
this proves the claim.
\end{rem}
 We now come to a non-commutative 
 generalization of our result from \cite{Pi3}.
\begin{thm}\label{t2} Let ${\cl M},{\cl N} $
 be   von Neumann algebras, with $\tau'$ as before.
Suppose $\Lambda_1=\{\psi^1_n\}$ and $\Lambda_2=\{\psi^2_n\}$ 
are two completely Sidon sets in 
a $C^*$-subalgebra $A\subset \cl M$.
Assume  
 there is a representation $\pi: A \to \cl N$
 such that for some $\d>0$ we have
$$\forall \psi \in \Lambda_1\cup\Lambda_2\quad \|\pi(\psi)\|_2\ge \d.$$
We assume  that $\pi(\Lambda_1)$ and $\pi(\Lambda_2)$ are mutually orthogonal in $L_2(\tau')$.
Then the union $ \Lambda_1\cup\Lambda_2$ is completely
$\otimes_{\max}^4$-Sidon.
\end{thm}
\begin{proof}  
We first observe that
since $\pi$ extends to a (normal) representation
from $A^{**}$ to $\cl N$, we may assume without loss
of generality that $\cl M=A^{**}$ and that $\pi$ is extended to $\cl M$.
Note that by our assumption
$\Lambda_1\cup\Lambda_2$ is bounded in $\cl M$.
By a simple homogeneity
argument, we may assume without loss of generality that
$$\forall \psi \in \Lambda_1\cup\Lambda_2\quad \|\pi(\psi)\|_2=1.$$
  By Lemma \ref{L8} 
 there are $x_n^1\in \cl M^*$
 biorthogonal to $(\psi^1_n)$ 
such that  $(x_n^1\otimes  \hat{\pi(\psi^1_n)})$
is free-gaussian dominated  in $( {\cl M}^{**} \ovl\otimes {\cl N})_*$.
Note that the latter is also  biorthogonal to $(\psi^1_n\otimes \hat{\pi(\psi^1_n)})$.
Similarly
there are $x_n^2\in \cl M^*$ such that the same holds for
   $(x_n^2\otimes  \hat{\pi(\psi^2_n)})$.
By Remark \ref{R9}, the union
 $\{x_n^1\otimes\hat{\pi(\psi^1_n)}\}\cup \{x_n^2\otimes \hat{\pi(\psi^2_n)}\}
 \subset ( {\cl M}^{**} \ovl\otimes {\cl N})_*$
 is free-gaussian dominated.
 But now the latter system
 is biorthogonal to 
 $\{\psi_n^1\otimes {\pi(\psi^1_n)} \}\cup \{\psi_n^2\otimes  {\pi(\psi^2_n)}\}\subset {\cl M}^{**}\ovl\otimes {\cl N}$.
 Indeed, this holds because,
 by our orthogonality assumption, $\hat{\pi\psi^1_n}(\pi \psi_n^2)=\hat{\pi\psi^2_n}( \pi\psi_n^1)=0$ for \emph{all}
 $m,n$. By Theorem \ref{t1}
we conclude that
 the latter system,
 which can be described as $([id:{\cl M}\to {\cl M} ^{**}] \otimes \pi)(\{\psi_n^1\otimes \psi_n^1\}\cup \{\psi_n^2\otimes \psi_n^2\})$,
  is completely
$\otimes_{\max}^2$-Sidon.
 Using \eqref{e8'} to remove $$[id:{\cl M}\to {\cl M} ^{**}] \otimes \pi: \cl M \otimes_{\max}  \cl M \to {\cl M }^{**}  \otimes_{\min}  \cl N \subset {\cl M} ^{**} \ovl\otimes {\cl N},$$ we see that this implies that
  $ \Lambda_1\cup\Lambda_2$ itself  is completely
$\otimes_{\max}^4$-Sidon.
\end{proof}
\begin{rem} As the reader may have noticed the preceding proof actually
shows that $\{\psi\otimes \psi  \otimes \psi \otimes \psi \mid \psi \in  \Lambda_1\cup\Lambda_2 \}$ is completely Sidon
in $(A\otimes_{\min} A) \otimes_{\max} (A\otimes_{\min} A)$.
\end{rem}
Let $G$ be any discrete group.
We say that subset $\Lambda\subset G$ is completely Sidon
if the 
set $\{U_G(t)\mid t\in \Lambda\}$ is   completely Sidon in $C^*(G)$.
In this setting we recover our recent generalization    \cite{Pifz} of Drury's classical commutative result. 
\begin{cor}\label{coru} Let $G$ be any discrete group.
 The union of two completely Sidon subsets of $G$
is completely Sidon.
\end{cor}
\begin{proof} We claim that any completely $\otimes^4_{\max}$-Sidon
set in $G$ is completely Sidon.
With this claim, the Corollary follows from
Theorem \ref{t2} applied with $A_1=A_2=C^*(G)$.
To check this claim, we use (ii) in Proposition \ref{R3}.
Let $U_G$ be the universal representation on $G$.
Assuming $\Lambda=\{t_n\}$. Let $\psi_n=U_G(t_n)$.
For any unitary representation
 $\pi$ on $G$ with values in a unital $C^*$ algebra $A_\pi$,
 with the same notation as in (ii), we have obviously
 (since $\pi$ extends completely contractively to $C^*(G)$)
$$\|\sum a_n\otimes \pi(t_n) \|_{M_k(A_\pi)}\le  \|\sum a_n\otimes U_G(t_n) \|_{M_k(C^*(G))} .$$
Applying this with $\pi=U_G\otimes U_G\otimes U_G\otimes U_G$,
and $A_\pi=C^*(G)\otimes_{\max} C^*(G) \otimes_{\max}C^*(G)\otimes_{\max} C^*(G)$,
the claim becomes immediate.
 \end{proof}
We refer to \cite{Pifz} for  several complementary results, in particular for ``completely Sidon" versions 
of the interpolation and Fatou-Zygmund properties of Sidon sets, and for a discussion
of the  closed span of a completely Sidon
 set  in the \emph{reduced} $C^*$-algebra of $G$. 

\begin{rem}\label{Rf} By analogy with the commutative case,
we propose the following definition: 
Let 
$(y_n)$ be a free-gaussian (i.e. free semicircular) sequence  in $M_*$. We say that $(x_n)$
in $A^*$ is free-subgaussian if there is $C$ such that
for any $k$  the union of the sequences $\{x_n^1\},\cdots,\{x_n^k\}$
in $(A^{\ast k})^*$
is $C$-dominated by 
$(y_n)$.
Here $A{\ast} \cdots \ast A$ is the (full) free product of $k$ copies of $A$,
and $x_n^1,\cdots,x_n^k$ are the copies of $x_n$ in each of the 
free factors of $A\ast \cdots \ast A$.
Note that with the same notation
the sequence  ${y_n^1,\cdots,y_n^k}$
in $(M{\ast} \cdots \ast M )^*$ has the same distribution as
the original sequence  $(y_n)$.\\
In the commutative case, when $(x_n)$ lies in $L_1$
over some probability space and freeness is replaced by independence, this is the same as subgaussian
in the usual sense,
see   \cite[Prop. 2.10]{Pi3} for details.
See \cite{Pis} for a survey on subgaussian systems.
\end{rem}

   \medskip
   
   \medskip
   
\n\textit{Acknowledgement.} Thanks are due  to  Marek Bo\.zejko, 
Simeng Wang 
and Mateusz Wasilewski   for useful communications.

  \end{document}